\documentclass[11pt]{amsart}
\usepackage{amsmath, amssymb, amsfonts, verbatim, amsthm, xcolor}
\usepackage{graphics, epsfig, enumerate, amstext,  comment}
\usepackage{graphicx, ulem, hyperref}
\newtheorem{theorem}{Theorem}[section]

\newtheorem{proposition}[theorem]{Proposition}
\newtheorem{corollary}[theorem]{Corollary}
\newtheorem{lemma}[theorem]{Lemma}
\newtheorem{example}[theorem]{Example}

\newtheorem{remark}[theorem]{Remark}

\def\qed{\hfill $\Box$\medskip}
\def\diag{{\rm diag}\,}
\def\Ker{\mathop{\rm Ker}}
\def\Ran{\mathop{\rm Ran}}
\def\bM{{\mathbb M}}
\def\bH{{\mathbb H}}

\def\cC{{\mathcal C}}
\def\Span{\mathop{\rm Span}}
\def\span{\mathop{\Span}\nolimits}

\def\cS{{\mathcal S}}

\def\IC{{\mathbb{C}}}
\def\IR{{\mathbb{R}}}
\def\ZZ{{\mathbb{Z}}}
\def\IT{{\mathbb{T}}}

\def\TT{{\mathbb{T}}}

\def\bV{{\bf V}}

\def\tr{\mathop{\rm tr}}

\def\bV{{\bf V}}

\def\IF{{\mathbb F}}

\def\conv{{\rm conv}}

\def\Tr{\mathop{\mathrm{Tr}}}

\begin{document}
\openup .95\jot

\title{Linear preservers of parallel matrix pairs with respect to the $k$-numerical radius}
\author[Kuzma]{Bojan Kuzma}
\author[Li]{Chi-Kwong Li}
\author[Poon]{Edward Poon}
\author[Singla]{Sushil Singla}
\maketitle

\begin{abstract}
 Let $1 \leq k < n$ be integers.  Two $n \times n$ matrices $A$ and $B$ form a parallel pair with respect to the $k$-numerical radius $w_k$ if $w_k(A + \mu B) = w_k(A) + w_k(B)$ for some scalar $\mu$ with $|\mu| = 1$; they form a TEA (triangle equality attaining) pair if the preceding equation holds for $\mu = 1$.  We classify linear bijections on $\bM_n$ and on $\bH_n$ which preserve parallel pairs or TEA pairs.  Such preservers are scalar multiples of $w_k$-isometries, except for some exceptional maps on $\bH_n$ when $n=2k$.
\end{abstract}



\section{Introduction}

Let $\bV$ be a normed space equipped with a norm $\| \cdot \|$ over $\IF = \IC$ or $\IR$.
Suppose $x, y \in \bV$. Then $x$ is parallel to $y$, denoted by $x\| y$,
 if $\|x + \mu y\| = \|x\| + \|y\|$ for  some $\mu\in \IF$ with $|\mu| = 1$.
 If the equality holds with $\mu=1$, we say that $x$ and $y$ form a triangle equality attaining (TEA) pair.
We are interested in linear maps $T\colon  \bV \rightarrow \bV$ preserving parallel pairs
i.e., $T(x)\| T(y)$ whenever $x\|y$. We are also interested in linear maps preserving TEA pairs, i.e.,
$(T(x), T(y))$ is a TEA pair whenever $(x,y)$ is a TEA pair.

Note that the zero map, and  a map of the form $x \mapsto f(x)y$ for a fixed $y\in \bV$
and a fixed linear functional $f$ on $V$ always preserves parallel pairs. Thus, the semigroup of
linear preservers of parallel pairs contains degenerate maps. Furthermore, it also contains all  scalar multiples of  isometries.
We are interested  if  it contains additional invertible linear maps or not. We remark that the answer
relies on the properties  of the given norm. For example, if the norm is strictly convex, i.e., $\|x+y\| < \|x\| + \|y\|$
whenever $x, y$ are linearly independent,
then $x$ and $y$ are parallel if and only if they are linearly dependent, in which case every 
linear map will preserve parallel pairs.
Nevertheless, there are results showing that bijective linear maps
preserving parallel pairs or TEA pairs must be nonzero multiples of isometries; see \cite{KLPS,LTWW0,LTWW}.
In this paper, we
consider parallel pair preservers and TEA preservers on $\bM_n$, the space of complex $n \times n$ matrices, when
the norm is the $k$-numerical radius for $1\leq k<n$. 
Recall that the $k$-numerical range and the $k$-numerical radius of $A \in \bM_n$ are defined as
$$W_k(A) = \{ \tr AP: P^* = P = P^2, \tr P = k\},$$
and
$$w_k(A) = \max\{ |\mu|: \mu \in W_k(A)\}.$$
It is known that $w_k$ is a norm on $\bM_n$ and is invariant under unitary similarities, i.e.,
$w_k(U^*AU) = W_k(A)$ for any $A \in \bM_n$ and unitary $U\in \bM_n$.
It follows from~\cite{Li87} (see also \cite{LiTsing88a}  for more general results)  that a linear map
$T\colon \bM_n\rightarrow \bM_n$ is an isometry for $w_k$ if and only if there is $\mu \in \IC$ with
$|\mu| = 1$ and a unitary $U$ such that one of the following holds.

\medskip
(1) $\mu T$ has the form $A \mapsto U^*AU$ or  $A \mapsto U^*A^tU,$

\medskip
(2)
$n=2k>2$ and $\mu T$ has the form
$A \mapsto \dfrac{1}{k}(\tr A)I-U^*AU$ or $A \mapsto \dfrac{1}{k}(\tr A)I-U^*A^tU.$

\medskip
We will confirm that invertible linear maps on $\bM_n$ preserving parallel pairs or TEA pairs with respect
to the $k$-numerical radius norm are scalar multiplies of isometries.
In particular, we have the following.

\begin{theorem} \label{main}
Let $n\ge 2$ and $1\leq k<n$, and $T\colon \bM_n \rightarrow \bM_n$ be a bijective linear map.
The following conditions are equivalent.
\begin{itemize}
\item[{\rm (a)}] $T$ preserves
parallel pairs with respect to the $k$-numerical radius.
\item[{\rm (b)}] $T$ preserves TEA pairs with respect to the $k$-numerical radius.
\item[{\rm (c)}]
There
is a nonzero $\mu \in \IC$ and a unitary $U \in \bM_n$ such that

\medskip
{\rm (c.1)} $\mu T$ has the form  $A \mapsto U^*AU$ or  $A \mapsto U^*A^tU,$ or

\medskip
{\rm (c.2)} $n=2k>2$ and $\mu T$ has the form
$A \mapsto \dfrac{1}{k}(\tr A)I-U^*AU$ or $A \mapsto \dfrac{1}{k}(\tr A)I-U^*A^tU.$
\end{itemize}
\end{theorem}

Let $\bH_n$ be the real linear space of Hermitian matrices in $\bM_n$.
We will show that a similar result holds for the real space of
Hermitian matrices except when $n = 2k$.

\begin{theorem}\label{main:Hermitian}
	Let $n \geq 2$ and $1 \leq k < n$. Suppose $T\colon\bH_n \to \bH_n$ is
	 a bijective real linear map. The following conditions are equivalent.
	 \begin{itemize}
	 \item[{\rm (a)}] $T$
	   preserves parallel pairs with respect to the $k$-numerical radius.
	 \item[{\rm (b)}] $T$
	   preserves TEA pairs with respect to the $k$-numerical radius.
	 \item[{\rm (c)}] There exist a nonzero $\mu \in \IR$ and  a unitary $U\in \bM_n$ such that one of the following holds.
	
\medskip
{\rm (c.1)} $\mu T$ has the form  $A \mapsto U^*AU$ or  $A \mapsto U^*A^tU.$

\medskip
{\rm (c.2)} $n=2k>2$ and $\mu T$ is a unital map such that

\medskip
\centerline{\ \quad
$T(A) =  c U^*AU$ for all $A \in \bH_n^0$, \ or
\
$T(A) = cU^*A^tU$ for all $A\in \bH_n^0$,}

\medskip
\quad where $\bH_n^0$ denotes the set of trace zero matrices in $\bH_n$,
and $c\in\IR\setminus\{0\}$.

\medskip
{\rm (c.3)}  $n = 2$ and $\mu T$ is a bijective unital trace preserving map.
\end{itemize}
\end{theorem}

The proofs of Theorem 1.2 and Theorem 1.3 will be given in Section 2 and Section 3, respectively.
We will always assume that $1 \leq k < n$ in our discussion.  Let us collect some basic properties of the $k$-numerical range
and $k$-numerical radius for easy reference in the next proposition. We will denote by
$\lambda_1(G) \ge \cdots \ge \lambda_n(G)$ the eigenvalues of  $G \in \bH_n$.

\begin{proposition}\label{prop:propertiesofW_k}
 The following properties of $W_k(A)$ and $w_k(A)$ hold.
 \begin{enumerate}
\item[{\rm (1)}] $W_k(\mu A + \nu I ) = \mu W_k(A) + k\nu$ for any $\mu, \nu \in \IC$.

\item [{\rm (2)}]
$W_k(A) = \bigcap_{\theta\in [0, 2\pi)} \{ \mu: e^{i\theta} \mu +
 e^{-i\theta} \bar \mu \le \sum_{j=1}^k \lambda_j(e^{i\theta}A +
e^{-i\theta} A^*)\}.$

If $A\in \bH_n$,  then $W_k(H)$ is the real interval $[\sum_{j=1}^k \lambda_{n+1-j} (A), \sum_{j=1}^k \lambda_j (A)]$.

\item[{\rm (3)}] $w_k(A) = \max_{\theta\in [0, 2\pi)} \sum_{j=1}^k \lambda_j(e^{i\theta}A +
e^{-i\theta} A^*)/2$.

\item [{\rm (4)}] Suppose $A = A_1 \oplus A_2 \in \bM_\ell\oplus \bM_{n-\ell}$. Then
$$W_k(A) = \conv\{\mu_1 + \mu_2: \mu_1 \in W_r(A_1), \mu_2\in W_{k-r}(A_2), 0 \le r \le \ell, 0 \le k-r \le n-\ell\}.$$
Consequently,
$$w_k(A) \leq \max\{ w_r(A_1) + w_{k-r}(A_2):
 0 \le r \le \ell, 0 \le k-r \le n-\ell\},$$
where $W_0(X) = \{0\}$ and $w_0(X) = 0$ for any square matrix $X$.
\item[{\rm (5)}] The set $W_k(A)$ is a nondegenerate line segment with endpoints $a_1, a_2$ if and only if
$A = cI + dH$ for some $c, d \in \IC$ with $|d| = 1$ and $H = H^*$ with eigenvalues $\mu_1 \ge \cdots  \ge \mu_n$
such that $a_1 = ck + d(\mu_1 + \cdots + \mu_k)$ and
$a_2 = ck + d(\mu_n + \cdots + \mu_{n-k+1})$. Also, $W_k(A) = \{\xi\}$ if and only if $A = \xi I/k$.
\end{enumerate}
\end{proposition}

Note that (1) appears in \cite[(2.1)]{Li94}; (2) follows from \cite[(5.1)]{Li94} and the convexity of $W_k(A)$ \cite[(3.1)]{Li94}.
For (4) we refer to  \cite[(4.3)]{Li94} (see also~\cite[Theorem 4.4]{LPW} for a  generalization) while (5) follows from property (1) and \cite[(5.1a)]{Li94}.

Our study concerns parallel pairs of matrices with respect to the  $k$-numerical radius.
We refer to a  recent paper~\cite{Zamani} for the study of parallel pairs with respect to the numerical radius in $C^*$ algebras
and \cite{GS} for the study of parallel pairs with respect to the joint numerical radius for $m$-tuples of matrices.

\section{Proof of Theorem \ref{main}}\label{section2}

We say that the $k$-numerical radius of $A$ is attained at a rank-$k$ projection $P$ if $|\tr AP| = w_k(A)$.
We begin with a few results concerning the properties of parallel pairs,
which are useful for our proofs and may be of independent interest.
We shall use the following lemma repeatedly without comment.

\begin{lemma} \label{lem1-new}
Suppose $A, B \in \bM_n$.
Then $A\|B$ if and only if any one of the following
equivalent conditions holds.
\begin{itemize}
\item[{\rm (a)}] There is a rank-$k$ orthogonal projection
$P$ such that $w_k(A) =|\tr(AP)|$ and $w_k(B) = |\tr(BP)|$.
\item[{\rm (b)}]
There is a unitary $U$ such that the leading $k\times k$ blocks $A_1, B_2$ of $U^*AU$ and $U^*BU$
satisfy  $w_k(A) = |\tr A_1|$ and $w_k(B) = |\tr B_1|$.
\end{itemize}
\end{lemma}

\begin{proof}
The statements follow readily from the definitions.
\end{proof}

For $A \in \bM_n$, let
$${\mathcal P}(A) = \{B \in \bM_n:\;  B \hbox{ is parallel to } A\},$$ and for a rank-$k$ orthogonal projection $P\in \bM_n$, let us define
$$S(P) = \{B \in \bM_n: \tr(BP) = w_k(B)\}.$$
By Lemma \ref{lem1-new}, we have the following.

\begin{lemma} \label{lem5-new}
Let $1 \le k < n$, $A \in \bM_n$.
Then
$${\mathcal P}(A) = \bigcup\{e^{i\theta}S(P): \theta \in [0, 2\pi), P^2 = P = P^*, \tr P = k,
|\tr AP| = w_k(A)\}.$$
\end{lemma}

Recall that a set $\cC\subseteq \bM_n$ is a convex cone if
$rA + sB \in \cS$ whenever $A, B \in \cC$ and $r,s \in [0, \infty)$.
The real affine dimension of $\cC$ is the (real) dimension of
$\Span_\IR(\cC)$, the linear span of $\cC$ over $\IR$.

\begin{lemma}  \label{lem4-new}
Suppose $P\in \bM_n$ is a rank-$k$ orthogonal projection. Then $S(P)$ is a convex cone with real affine dimension
$$\gamma(n,k) =   k^2+(n-k)^2 + n^2 -1.$$
Moreover, with respect to the direct sum decomposition $\IC^n = \Ran P \oplus \Ker P$,
\begin{equation}\label{E:form}
S(P) = \left\{ B = \begin{pmatrix} B_{11} & B_{12} \\ - B_{12}^* & B_{22} \end{pmatrix} :
B_{11} \in \bM_k, B_{22} \in \bM_{n-k}, B_{12}\in\bM_{k\times (n-k)}, \tr B_{11} = w_k(B) \right\}
\end{equation}
and
$$\span_{\IR}(S(P))=  \left\{\begin{pmatrix} B_{11} & B_{12} \\ - B_{12}^* & B_{22} \end{pmatrix} :
B_{11} \in \bM_k, B_{22} \in \bM_{n-k}, B_{12}\in\bM_{k\times (n-k)}, \tr B_{11} \in \IR \right\}.$$
\end{lemma}

\begin{proof}
Note that $B=H+i G \in S(P)$, with $H,G$ hermitian, if and only $\tr(PHP)+i\tr(PGP)=w_k(B)$.  Thus, $\tr(PGP)=0$.  	
Moreover, $w_k(B) \ge w_k(H) =  \sum_{j=1}^k \lambda_j(H) \ge \tr(PHP) = w_k(B)$ shows that $\tr(PHP) = \sum_{j=1}^k \lambda_j(H)$.
By  \cite[Lemma 4.1]{Li}, $H = H_1 \oplus H_2 \in \bH_k \oplus \bH_{n-k}$ is block diagonal with respect to the  direct sum
decomposition $\IC^n = \Ran P \oplus \Ker P$, and $\lambda_k(H_1) \geq \lambda_1(H_2)$.  Thus $B \in S(P)$ if and only if it has the form
in \eqref{E:form}.
	
Let $B,C \in S(P)$ and $b,c \geq 0$.  Then $G = bB + cC$ has the form
$\begin{pmatrix} bB_{11} + cC_{11} & G_{12} \cr -G_{12}^* & G_{22} \cr\end{pmatrix}$ and
$$w_k(G) \geq \tr GP = b \tr B_{11} + c \tr C_{11} = b w_k(B) + cw_k(C) \geq w_k(G),$$
whence $G \in S(P)$.  Thus $S(P)$ is a convex cone.

The affine dimension of $S(P)$ is the dimension of the real linear subspace spanned by $S(P)$. By \eqref{E:form}, $S(P) \subseteq S_1 + iS_2$, where
 $S_1 = \span_{\IR} \{H_1 \oplus H_2: H_1 = H_1^* \in \bM_k,
H_2 = H_2^* \in \bM_{n-k}\}$ has dimension $k^2 + (n-k)^2$, and
$S_2 = \span_{\IR} \{ G\in \bM_n: G = G^*, \tr PGP = \sum_{j=1}^k G_{jj} = 0\}$ has dimension $n^2-1$.
So, $\dim \span_{\IR} S(P) \le n^2 + k^2 + (n-k)^2 - 1 = \gamma(n,k)$.
It suffices to show that $S_1 + iS_2 \subseteq \span_{\IR} S(P)$.

{\bf Case 1.} Let $A = H_1 \oplus A_2$ such that
 $H_1^* = H_1 \in \bM_{k}$ is positive-definite and
 $A_2 \in \bM_{n-k}$
 satisfies $\lambda_k(H_1) \ge \|A_2\|$.
 Then by property (4) in  Proposition~\ref{prop:propertiesofW_k}, we see that
 $w_k(A) = \tr H_1$ so that
 $A \in S(P)$. Consequently, $\bH_k\oplus\bM_{n-k}\subseteq\span_{\IR} S(P)$.

{\bf
Case 2.}  Suppose $1 \le i \le k < j \le n$ and
$A = (I_k\oplus 0_{n-k}) + i (\mu E_{ij} + \bar \mu E_{ji})$ for $\mu \in \IC$ with $|\mu| = 1/2$. Then $A$ is permutationally similar to $A_0 \oplus I_{k-1} \oplus 0_{n-k-1}$, where $A_0 = \begin{pmatrix} 1 & i\mu  \cr i\bar \mu & 0\cr\end{pmatrix}$.  Note that $W_{1}(A_0)$ is a circular disk of radius $1/2$ centered at $1/2$. Then by property (4) in  Proposition~\ref{prop:propertiesofW_k},
$$w_k(A) \leq w_1(A_0) + w_{k-1}(I_{k-1} \oplus 0_{n-k-1}) = 1 + (k-1) = \tr (AP) \leq w_k(A),$$
whence $A \in S(P)$.
Combining this result with Case 1 shows that $i G\in\span_{\IR} S(P)$ for every Hermitian $G$ of the form $\left(\begin{smallmatrix}
		0 & G_{12} \\
		G_{12}^\ast & 0
\end{smallmatrix}\right)$ with $G_{12}\in\bM_{k\times(n-k)}.$

{\bf
Case 3.}  Suppose $k > 1$.  Let $0 \ne C \in \bM_k$ satisfy $\tr C = 0$.  Let $\epsilon = 1/(2k\|C\|)$ and let $A = (I_k + \epsilon C) \oplus 0_{n-k}$. Note that for any rank-$r$ orthogonal projection $Q \in \bM_k$ with $r<k$, $$|\tr((I_k+\epsilon C)Q)| \leq |\tr Q| + \epsilon |\tr (CQ)| \leq r + \epsilon r \|C\| < r + 1/2 < k = \tr(AP).$$
Then by property (4) in Proposition~\ref{prop:propertiesofW_k},
\begin{eqnarray*}
w_k(A) &\leq& \max_{0 \leq r \leq k} w_r(I_k + \epsilon C) \\
&=& \max \{|\tr (Q + \epsilon CQ)| : 0 \leq r \leq k, Q^2=Q^*=Q \in \bM_k, \tr Q = r\} \leq \tr (AP),
\end{eqnarray*}
so $A \in S(P)$.  Combining this result with Case 1 shows that $C \oplus 0_{n-k} \in \span_{\IR}S(P)$ for every $C \in \bM_k$ with $\tr C = 0$.

Finally, cases 1 - 3 together show that $S_1 + i S_2 \subseteq \span_{\IR}S(P)$, as desired.
\end{proof}

In the following discussion, we always let
$$\gamma(n,k) =   k^2+(n-k)^2 + n^2 -1.$$
Notice that two rank-$k$ projections  $P_1,P_2$ coincide if and only if $S(P_1)=S(P_2)$.
 We shall show more in the next lemma.

\begin{lemma} \label{lem-2proj}
  Given two rank-$k$ orthogonal projections $P,Q$ and real numbers $\theta_1,\theta_2$, we have
  $$\span_{\IR}e^{i\theta_1} S(P)=\span_{\IR} e^{i\theta_2}S(Q)$$
  if and only if $P=Q$ and $\theta_1=\theta_2$ modulo $\pi$.
\end{lemma}

\begin{proof}
If $P\neq Q$, then there exists a unit vector $x \in \Ker P$ with $x \notin \Ker Q$. Let $G = xx^*$.  By Lemma~\ref{lem4-new}, $e^{i\theta} G \in\span_{\IR}S(P)$ for all $\theta \in \IR$, so $e^{i\theta} G \in\span_{\IR}e^{i\theta_1}S(P)$ for all $\theta \in \IR$. On the other hand, $x\notin  \Ker Q$ implies $Qx\neq0$, and hence  $\tr(QGQ)=\|Qx\|^2 \ne  0$.
Take $\theta = \pi/4 + \theta_2$.  Then
$$e^{i(\theta - \theta_2)} G= \cos(\pi/4)G+i \sin(\pi/4)G\notin\span_{\IR}S(Q)$$
by Lemma~\ref{lem4-new} because $\tr Q(e^{i(\theta - \theta_2)} G)Q = e^{i(\theta - \theta_2)}\|Qx\|^2 \notin \IR$.
Hence $e^{i\theta} G\notin \span_{\IR} e^{i\theta_2} S(Q)$, so the spans are different.

Lastly, assume $P=Q$ but $\theta_1-\theta_2\notin\ZZ\pi$. Consider any unit vector
$x\in\Ran P$.
By Lemma~\ref{lem4-new} $H:=xx^\ast$ belongs to $\span_{\IR} S(P)$, so $e^{i\theta_1}H\in\span_{\IR}e^{i\theta_1}S(P)$. However,
 $$e^{i(\theta_1-\theta_2)}H=\cos(\theta_1-\theta_2)H+ i\sin(\theta_1-\theta_2)H$$ has nonzero imaginary part with $\tr(P\sin(\theta_1-\theta_2)HP)=\sin(\theta_1-\theta_2)\|x\|^2\neq0$, and as such does not belong to $\span_{\IR}S(P)$. Hence, $e^{i\theta_1}H\notin\span_{\IR}e^{i\theta_2}S(P)$.
\end{proof}
\begin{lemma}\label{conescontainment}
Let $T\colon \bM_n \rightarrow \bM_n$ be a bijective $\IR$-linear map.
 Suppose $P_1$, $P_2$ are distinct rank-$k$ orthogonal projections.
Then there do not exist a rank-$k$ orthogonal projection $Q$ and $\theta_3 \in \IR$ such that $T(e^{i\theta_1}S(P_1))\cup T(e^{i\theta_2}S(P_2))\subseteq e^{i\theta_3}S(Q)$.
\end{lemma}
\begin{proof}
By Lemma \ref{lem4-new}, $e^{i\theta}S(P)$ contains $\gamma(n,k)$ $\IR$-linearly independent matrices for any rank-$k$ projection $P$ and any choice of $\theta \in \IR$.
Since $P_1$ and $P_2$ are distinct, Lemma \ref{lem-2proj} implies that
$e^{i\theta_1}S(P_1)\cup e^{i\theta_2}S(P_2)$ contains at least $\gamma(n,k)+1$ real linearly independent matrices,  and hence, by the bijectivity of $T$, so does  $T(e^{i\theta_1}S(P_1))\cup T(e^{i\theta_2}S(P_2))$.  By Lemma \ref{lem4-new}, this exceeds the real affine dimension of $\span_{\IR} e^{i\theta_3} S(Q)$, so the result follows.
\end{proof}

Let $\IT = \{\mu \in \IC: |\mu| = 1\} \subseteq\IC$ denote the unimodular group.

\begin{lemma}\label{L:single_cone}
Let $A \in \bM_n$.  Suppose $w_k(A)$ is attained at a unique (rank-$k$) projection $P$, $w_k(A) = \tr (AP)$, and $\Ran P$ is a reducing subspace for $A$.
Suppose $\cC$ is a cone with affine dimension greater or equal to $\gamma(n,k) = k^2 + (n-k)^2 + n^2 -1$ satisfying $A \in \cC \subseteq \mathcal{P}(A)$.  Then $\cC \subseteq S(P)$, so its affine dimension equals $\gamma(n,k)$.
\end{lemma}

\begin{proof}
With respect to the decomposition $\IC^n = \Ran P \oplus \Ker P$, we may write $A = \begin{pmatrix} A_{11} & 0 \\ 0 & A_{22} \end{pmatrix}$ with $A_{11} \in \bM_k$, $A_{22} \in \bM_{n-k}$, and $\tr A_{11} =w_k(A) > 0$.  Note that, by Lemma \ref{lem1-new}, $\mathcal{P}(A) = \TT S(P)$.
By Lemma \ref{lem4-new}, every nonzero element of $\cC$ has the form $\mu B = \mu \begin{pmatrix} B_{11} & B_{12} \\ -B_{12}^* & B_{22} \end{pmatrix}$ with $\tr B_{11} =w_k(B) > 0$ and $\mu \in \TT$.
\medskip

We first show that if $B_{12} \ne 0$, then $\mu = 1$.

Since $\cC$ is a cone, for every $r > 0$
$A + r \mu B \in \cC \subset \TT S(P)$, so
\begin{equation}\label{E:block}
\begin{pmatrix} A_{11} + r \mu B_{11} & r \mu B_{12} \\ -r \mu B_{12}^* & A_{22} + r \mu B_{22} \end{pmatrix} = \nu \begin{pmatrix} C_{11} & C_{12} \\ - C_{12}^* & C_{22} \end{pmatrix} = \nu C
\end{equation}
for some $\nu \in \TT$ and some $k \times (n-k)$ matrix $C_{12}$ (both of which may depend on $r$).  Equating the $(2,1)$-block entries and taking the conjugate transpose gives $-r \bar{\mu} B_{12} = -\bar{\nu} C_{12}$, so
$$ \nu^2 r \bar{\mu} B_{12} = \nu C_{12} = r \mu B_{12}$$
(by equating the $(1,2)$-block entries in \eqref{E:block}), whence $\nu = \pm \mu$.

Equating the $(1,1)$-block entries in \eqref{E:block} and taking traces gives
$\tr A_{11} + r \mu \tr B_{11} = \pm \mu \tr C_{11}$.  Because $\tr A_{11}$, $\tr B_{11}$, $\tr C_{11}$ are all positive, $\mu = \pm 1$.  If $\mu = -1$ we can choose $r = \tr A_{11}/\tr B_{11}$ to get $0 = \pm \tr C_{11} = \pm w_k(C) \ne 0$, a contradiction.  So $\mu = 1$ as claimed.

If  $0 \ne  D:=D_{11} \oplus D_{22} \in \cC$ and $\cC$ contains a matrix $B$ with $B_{12}=PB(I-P)\neq 0$, then the above arguments applied to  $D+r B\in\cC$ ($r>0$), give that $\tr (D_{11}+r B_{11})>0$ for every $r>0$, and hence $\tr D_{11}\ge0$; thus $D \in S(P)$ by Lemma \ref{lem4-new}.

Consequently $\cC \subset S(P)$, unless one never has $B_{12} \ne 0$, that is, if $\cC \subset \bM_k \oplus \bM_{n-k}$.  But the latter is impossible: since
$$\dim_{\IR} \bM_k \oplus \bM_{n-k} = 2k^2 + 2(n-k)^2 < k^2 + (n-k)^2 + n^2 -1 = \gamma(n,k),$$
$\cC$ is not contained in $\bM_k \oplus \bM_{n-k}$.
\end{proof}

We have established  basic properties for parallel pairs of matrices with respect to the $k$-numerical radius.
Now, we focus on bijective linear maps preserving parallel pairs.  The following lemma shows that
$T(I)$ is a scalar multiple of $I$; see \cite[Lamma 4.3]{Lietal}. We give a short proof of the result.

\begin{lemma} \label{lem2-new}
Let $A \in \bM_n$. The following are equivalent.
\begin{itemize}
\item[{\rm (a)}] $A = \gamma I$ for some $\gamma \in \IC$.
\item[{\rm (b)}]
$A\| B$ for every $B \in \bM_n$.
\item[{\rm (c)}] $A \| P$  for every rank-$k$ orthogonal projection $P \in \bM_n$.
\end{itemize}
\end{lemma}

\begin{proof}
  Suppose (a) holds and  $A = \gamma I$ with
$\gamma \in \IC$.  Note for any $B \in \bM_n$, there is a rank-$k$ projection
 $P$ such that $|\tr(BP)| = w_k(B)$. So there is
a complex unit $\xi$ such that
$$w_k(A + \xi B) \geq |\tr (\gamma I + \xi B) P | = |\gamma \tr P| + |\tr(BP)| = w_k(A) + w_k(B).$$
Thus, $B \in {\mathcal P}(A)$. Hence (b) holds.
The implication (b) $\Rightarrow$ (c) is clear.

Suppose (c) holds. Let $P$ be a rank-$k$ orthogonal projection. 
By Lemma \ref{lem1-new}, there exists a rank-$k$ orthogonal projection $Q$ satisfying $w_k(A) = |\tr AQ|$ and $k = w_k(P) = |\tr PQ| = \tr PQ$, forcing $Q = P$ (see~\cite[Lemma 4.1]{Li}).  Thus $|\tr AP| = w_k(A)$ for every rank-$k$ orthogonal projection $P$, so $W_k(A)$ lies on a circle.  Since $W_k(A)$ is convex (see~\cite{Li94} or~\cite{Westwick}), $W_k(A)$ must be a singleton and the result follows by property (5) of Proposition~\ref{prop:propertiesofW_k}.
\end{proof}

\begin{lemma}\label{lem6a--new} Let $T\colon \bM_n \rightarrow \bM_n$ be a bijective linear map preserving parallel pairs with respect to the $k$-numerical radius.  Suppose $A$  attains its $k$-numerical radius at a unique projection $Q$, and suppose $\Ran Q$ is reducing for $A$. Then $X=T^{-1}(A)$ also attains its $k$-numerical radius at a unique projection.
\end{lemma}

\begin{proof}  Let a unimodular $\mu\in\IC$ be  such that $\tr [(\mu A)Q]=w_k(A)=w_k(\mu A)$; clearly,  $\Ran Q$ is reducing also for $\mu A$.  Suppose, to reach a contradiction, that  $P_1$, $P_2$ are distinct rank-$k$ orthogonal projections such that $e^{-i\theta_j}\tr XP_j = w_k(X)$ for some $\theta_j \in \IR$ with $j=1,2$.
By the definition of $S(P)$ and Lemma \ref{lem5-new}, we have $X\in e^{i\theta_j}S(P_j)\subseteq {\mathcal P}(X)$. Applying $T$ gives
$$A\in T(e^{i\theta_j}S(P_j))\subseteq T({\mathcal P}(X))\subseteq \mathcal{P}(T(X)) = {\mathcal P}(A).$$
By Lemma \ref{lem4-new}, the cone $e^{i\theta}S(P)$ contains $\gamma(n,k)$ $\IR$-linearly independent matrices for any rank-$k$ projection $P$ and any choice of $\theta \in \IR$. Thus, by bijectivity of $T$, $T(e^{i\theta_{j}}S(P_j))$ is a cone of affine dimension $\gamma(n,k)$ contained in ${\mathcal P}(A)$.  As such,  $\mu T(e^{i\theta_{j}}S(P_j))$ is also a cone of dimension $\gamma(n,k)$ contained in   $\mu{\mathcal P}(A)={\mathcal P}(\mu A)$, and which contains $\mu A$.
Hence,   by Lemma~\ref{L:single_cone}, we have
$$\mu T(e^{i\theta_{j}}S(P_j))\subseteq S(Q).$$
This contradicts Lemma \ref{conescontainment}.
\end{proof}

\begin{lemma}\label{mainlemma}	Suppose $T\colon \bM_n \rightarrow \bM_n$ is a bijective unital linear map
preserving parallel pairs with respect to the $k$-numerical radius.
If $A \in \bM_n$ satisfies $W_k(A) = [-1,1]$ and $T(X) = A$, then $W_k(X) = [-1,1]$.
\end{lemma}

\begin{proof} Since $W_k(A) = [-1,1]$, we see that $A = A^*$ by Proposition \ref{prop:propertiesofW_k} (5).
Assume that $A$ has eigenvalues $a_1 \ge \cdots \ge a_n$ with corresponding
orthonormal eigenvectors $v_1, \dots, v_n$, and $T(X) = A$.
We consider two cases.

{\bf Case 1} Suppose $a_k > a_{k+1}$, $a_{n-k} > a_{n-k+1}$.

We have  $\sum_{j=1}^k a_j = 1 = - \sum_{j=1}^k a_{n+1-j}$.
Let $P_+$  be the rank-$k$ projection whose range is spanned by $\{v_1, \dots, v_k\}$.
Let $\gamma$ be an arbitrary complex number with $\mathrm{Re}\,\gamma > 0$.
By Proposition \ref{prop:propertiesofW_k} (1), $W_k(A+\gamma I) = W_k(A) + k \gamma = [-1 + k \gamma, 1+ k \gamma]$.
Thus $w_k(A+\gamma I)$ is attained at precisely those rank-$k$ projections $P$ satisfying $\tr AP = 1$.
By~\cite[Lemma 4.1]{Li} and Schur's result that the sum of the $k$ largest diagonal entries of a hermitian matrix is bounded above
by the sum of its $k$ largest eigenvalues~\cite[Theorem B.1, p.~300]{Marsh-Olkin} (or by the interlacing property for the eigenvalues
of a compression of hermitian matrix) we see that $P=P_+$ is the unique rank-$k$ projection satisfying $|\tr(P(A+\gamma I)P)|=w_k(A+\gamma I)=|1+k\gamma |$.   We may write
$w_k(A + \gamma I) = e^{i\theta_{\gamma}} \tr (A+ \gamma I)P_+$ for some $\theta_{\gamma} \in \IR$.
	
	Suppose $w_k(X + \gamma I)$ is attained at a rank-$k$ projection $Q_{\gamma}$. Then, by Lemma~\ref{lem5-new},
\begin{equation}\label{eq:Qgamma}
  X + \gamma I \in e^{i \phi_{\gamma}} S(Q_{\gamma}) \subseteq  \mathcal{P}(X + \gamma I)
\end{equation} for some $\phi_{\gamma} \in \IR$.  Applying $T$ we have
	$$A + \gamma I = T(X + \gamma I) \in T(e^{i \phi_{\gamma}} S(Q_{\gamma})) \subseteq T(\mathcal{P}(X+\gamma I ))
	\subseteq \mathcal{P} (T(X+\gamma I)) = \mathcal{P} (A + \gamma I).$$
	Multiplying by $e^{i\theta_{\gamma}}$ gives
$$e^{i\theta_{\gamma}} (A+\gamma I) \in e^{i(\theta_{\gamma} + \phi_{\gamma})} T(S(Q_{\gamma})) \subseteq
\mathcal{P} (e^{i\theta_{\gamma}} (A+\gamma I))=\mathcal{P}  (A+\gamma I).$$
Since $w_k(e^{i\theta_{\gamma}} (A+\gamma I)) = \tr (e^{i\theta_{\gamma}} (A+\gamma I) P_+)$ is attained uniquely at $P_+$,
and since $\Ran P_+$ reduces $e^{i\theta_{\gamma}} (A+\gamma I)$, and $e^{i (\theta_{\gamma} + \phi_{\gamma})} T(S(Q_{\gamma}))$
is a cone of dimension $\gamma(n,k)$, we can apply Lemma \ref{L:single_cone} to conclude that
$e^{i (\theta_{\gamma} + \phi_{\gamma})} T(S(Q_{\gamma})) = T(e^{i (\theta_{\gamma} + \phi_{\gamma})}S(Q_{\gamma})) \subseteq S(P_+)$.

But this holds for all $\gamma$ with positive real part.  Hence, by Lemma~\ref{conescontainment},
there is a single projection $Q_+$ such that $Q_{\gamma} = Q_+$ for all $\gamma$ with $\text{Re}\, \gamma > 0$.
Let $z_0 = \tr XQ_+ \in W_k(X)$.  Then
$$w_k(X + \gamma I) = |\tr (X+\gamma I) Q_+| = |\gamma k  + \tr XQ_+| = |\gamma k + z_0|$$
for all $\gamma$ with positive real part; by continuity, this holds for all $\gamma$ with nonnegative real part.  In particular, if
$w=w_1+iw_2$ is the Cartesian decomposition of $w\in W_k(X)$ (and likewise for $\gamma=\gamma_1+i\gamma_2$ and $z_0=x_0+iy_0$), we have
$$(w_1+k\gamma_1)^2+(w_2+k\gamma_2)^2\le (x_0+k\gamma_1)^2+(y_0+k\gamma_2)^2;\qquad \gamma_1\ge0.$$
When $\gamma_1=0$ this simplifies into the inequality
$$w_1^2-x_0^2\le (y_0+k\gamma_2)^2-(w_2+k\gamma_2)^2=(y_0-w_2) (2k \gamma_2 +w_2+y_0),$$
valid for every $\gamma_2\in\IR$, which is possible only if $w_2=y_0$. Thus $W_k(X)$ lies in the
horizontal line segment $[-x_0 + iy_0\,,\, x_0 + iy_0]$. It follows that the $k$-numerical range of  $\hat{X}:=X-\frac{y_0}{k} i I_n$
lies in a real line, so $\hat{X}$ is Hermitian by \cite[5.1(a)]{Li94}. Moreover, by the uniqueness of the rank-$k$ projection
$Q_+=Q_\gamma$ in~\eqref{eq:Qgamma},
we must have $\lambda_k(\hat{X}) > \lambda_{k+1}(\hat{X})$, where as usual we organize the eigenvalues of $\hat{X}$ in descending
order $\lambda_1(\hat{X}) \geq \dots \geq \lambda_n(\hat{X})$, and $Q_+$ is the projection onto the subspace spanned by the
eigenvectors of $\hat{X}$ corresponding to $\lambda_1(\hat{X}), \dots, \lambda_k(\hat{X})$.

Now let $P_-$  be the rank-$k$ projection whose range is spanned by $\{v_{n-k+1}, \dots, v_n\}$.
Then for all complex $\gamma$ with $\mathrm{Re}\,\gamma<0$, $P = P_-$ is the only rank-$k$ projection satisfying
$|\tr(P(A+\gamma I)P)|=w_k(A+\gamma I)=|1+k\gamma |$.
Hence, we may repeat the preceding arguments to conclude that there exists a unique projection $Q_-$ for which \eqref{eq:Qgamma}  holds
true when $\mathrm{Re}\,\gamma<0$.  Consequently $\lambda_{n-k}(\hat{X}) > \lambda_{n-k+1}(\hat{X})$ and $Q_-$ is the projection onto the
subspace spanned by the eigenvectors of $\hat{X}$ corresponding to $\lambda_{n-k+1}(\hat{X}), \dots, \lambda_n(\hat{X})$.

To summarize, $|\tr (\hat{X} \pm rI) Q_{\pm}| = w_k(\hat{X} \pm rI)$ for all $r >0$, so by continuity,
$w_k(\hat{X}) = |\tr \hat{X} Q_+| = |\tr \hat{X} Q_-|$ is attained at both $Q_+$ and $Q_-$.
Since $\lambda_k(\hat{X}) > \lambda_{k+1} (\hat{X})$,
$$\tr \hat{X} Q_+ = \sum_{j=1}^k \lambda_j(\hat{X}) > \sum_{j=1}^k \lambda_{n+1-j}(\hat{X}) = \tr \hat{X} Q_-,$$
so $\tr \hat{X} Q_- = - \tr \hat{X} Q_+$. This shows that $W_k(X)=[-b+iy_0,b+iy_0]$, for some $b>0$, is symmetric
relative to imaginary axis.

Finally, let $z \in \IC$ satisfy Re\,$z > 0$.  Then $X+zI \in (b+iy_0 + kz) S(Q_+)$.  Applying $T$, and noting that
$T(S(Q_+)) \subseteq \mu S(P_+)$ for some $\mu \in \TT$, gives
$$A + zI \in (b+iy_0 + kz) T(S(Q_+)) \subset (b+iy_0 + kz) \mu S(P_+).$$
Divide by $1+kz$ to get
$$\frac{A + zI}{1+kz} \in \frac{b+iy_0 + kz}{1+kz} T(S(Q_+))
\subseteq \frac{b+iy_0 + kz}{1+kz} \mu S(P_+) \subseteq \mathcal{P}  \left( \frac{A+zI}{1+kz}\right).$$
Since $w_k ((A+zI)/(1+kz)) = \tr P_+ (A+zI)/(1+kz)$ is attained uniquely at $P_+$, we can apply Lemma 2.5 to
conclude that the cone $\cC = \frac{b+iy_0 + kz}{1+kz} \mu S(P_+) \subseteq S(P_+)$.  Thus $f(z) = \frac{b+iy_0 + kz}{1+kz}$
has constant argument for Re\,$z >0$.  In other words, the M\"{o}bius transformation $f$
maps the open right-half plane into a ray, so it must be a constant function, whence $b+iy_0 = 1$.
Thus, $W_k(X) = [-1,1]$ as asserted.

{\bf Case 2.}
Suppose
$a_k = a_{k+1}$ or $a_{n-k} = a_{n-k+1}$.
For $m = 1,2,\dots$, let $A_m = (1-1/m)A + B/m$, where $B = \sum_{j=1}^n \beta_j v_jv_j^*$
with $\beta_1 > \cdots > \beta_n$, $\sum_{j=1}^k \beta_j = 1 = - \sum_{j=1}^k \beta_{n-j+1}$.
Then $\lambda_k(A_m) > \lambda_{k+1}(A_m)$, $\lambda_{n-k}(A_m) > \lambda_{n-k+1}(A_m)$,
$W_k(A_m) = [-1,1]$ and  $A_m\rightarrow A$. Suppose $T(X_m) = A_m$. Then
$W_k(X_m) = [-1,1]$ by the result  in Case 1.
Since $A_m \rightarrow A$, we see that
$X_m \rightarrow X$ so that $\lambda_j(X_m) \rightarrow \lambda_j(X)$ for $j = 1, \dots, n$.
Hence, $W_k(X) = \left[\sum_{j=1}^k \lambda_{n-j+1}(X), \sum_{j=1}^k \lambda_j(X)\right] = [-1,1]$.
\end{proof}

\textit{Proof of Theorem \ref{main}}.
The implications (c) $\Rightarrow$ (b) $\Rightarrow$ (a) are clear.
It remains to prove (a) $\Rightarrow$ (c). To this end, suppose $T\colon \bM_n \to \bM_n$ is a bijective linear preserver
of parallel pairs with respect to the $k$-numerical radius.  By Lemma \ref{lem2-new}, $T(\IC I) = \IC I$.
Without loss of generality we may assume $T$ is unital. By Lemma \ref{mainlemma},
$T^{-1}$ maps the set $\{A \in \bM_n : W_k(A) = [-1,1]\}$ to itself.
Thus, if $A = A^*$ and $W_k(A) = [\alpha_1, \alpha_2]$ with $\alpha_2 > \alpha_1$,
then for $r = (\alpha_2 + \alpha_1)/(2k)$ and $s = (\alpha_2-\alpha_1)/2$,
the matrix  $B = (A-r I)/s$  satisfies $W_k(B) = [-1,1]$, so $W_k(T^{-1}(B)) = [-1,1]$.
As a result, $W_k(T^{-1}(A)) = [\alpha_1, \alpha_2]$.
Thus, $T^{-1}$ preserves the $k$-numerical range of Hermitian $A$.
By \cite[Theorem 2]{Li87}, $T^{-1}$ has the desired form, and hence so does $T$.
\qed

Note that if $T^{-1}$ is a unital linear map sending the set $\{A \in \bM_n : W_k(A) = [-1,1]\}$ to itself, then
$T(\bH_n) = \bH_n$. One may then use Theorem \ref{main:Hermitian}
to finish the proof. Nevertheless, it is more direct
to show that $T^{-1}$ preserves the $k$-numerical range, so that there is no need to consider maps of the form (c.2) and (c.3) in Theorem \ref{main:Hermitian}.

\section{Proof of Theorem \ref{main:Hermitian}}\label{section3}

In this section, we use the following notations. For $A \in \bH_n$, let
$${\mathcal P}(A) = \{B \in \bH_n:\;  B \hbox{ is parallel to } A\},$$
and for a rank-$k$ orthogonal projection $P\in \bH_n$ let
$$S(P) = \{B \in \bH_n: \tr(BP) = w_k(B)\}.$$
Many of the lemmas for $\bM_n$ in Section 2 can be adapted to the Hermitian case.
We list the results below and only give the proofs if they are different.

\begin{lemma} \label{section3:lem1-new}
Suppose $A, B \in \bH_n$.
Then $A\|B$ if and only if there is a rank-$k$ orthogonal projection
$P$ such that $w_k(A) =|\tr(AP)|$ and $w_k(B) = |\tr(BP)|$.
\end{lemma}

Then from Lemma \ref{section3:lem1-new}, we get the following.

\begin{lemma} \label{section3:lem5-new}
Let $1 \le k < n$, $A \in \bH_n$.
Then
$${\mathcal P}(A) = \bigcup\{\theta S(P):\;\; \theta\in\{-1,1\}, P^2 = P = P^*, \tr P = k,
|\tr AP| = w_k(A)\}.$$
\end{lemma}

We fix $$\gamma^{\bH}_n(k) :=   k^2+(n-k)^2 .$$

\begin{lemma}  \label{section3:lem4-new}
Suppose $P\in \bH_n$ is a rank-$k$ orthogonal projection. Then $S(P)$ is a convex cone with real affine dimension $\gamma^{\bH}_n(k)$. Moreover, with respect to the direct sum decomposition $\IC^n = \Ran P \oplus \Ker P$,
\begin{equation}\label{E:form1}
S(P) = \left\{ B = \begin{pmatrix} B_{11} & 0 \\ 0 & B_{22} \end{pmatrix} : B_{11} \in \bH_k, B_{22} \in \bH_{n-k},  \tr B_{11} = w_k(B) \right\}
\end{equation}
and
$$\span_{\IR}(S(P))=  \bH_k\oplus\bH_{n-k}.$$
\end{lemma}

Notice that two rank-$k$ projections  $P_1,P_2$ coincide if and only if $S(P_1)=S(P_2)$; in fact, we shall show more in the next lemma, analogous to Lemma \ref{lem-2proj}. But there is a slight change in both statement and proof, so we provide a proof for the lemma below.

\begin{lemma} \label{section3:lem-2proj}
  Given two rank-$k$ orthogonal projections $P,Q$ we have
  $$\span_{\IR} S(P)=\span_{\IR} S(Q)$$
  if and only if $P=Q$, or $n=2k$ and $P=I-Q$.
\end{lemma}

\begin{proof}
If $P\neq Q$ and $P\neq I-Q$, then there exists a unit vector $x \in \Ker P$ with $x \notin (\Ker Q)\cup(\mathrm{Im}\, Q)$. Let $G = xx^*$.  By Lemma~\ref{section3:lem4-new}, $G \in\span_{\IR}S(P)$. On the other hand, $x\notin  \Ker Q$ and $x\notin\mathrm{Im}Q$  implies $xx^\ast \notin \bH(\mathrm {Im}Q) \oplus \bH(\Ker Q)=\span_{\IR} S(Q)$, where $\bH(V)$ denotes the space of all self-adjoint matrices on a subspace $V\subseteq\IC^n$.

The converse is trivial.
\end{proof}

\begin{lemma}\label{section3:conescontainment}
Let $T\colon \bH_n \rightarrow \bH_n$ be a bijective $\IR$-linear map.
 Suppose $P_1$, $P_2$ are  rank-$k$ orthogonal projections with $\span_{\IR} S(P_1)\neq\span_{\IR} S(P_2)$, and let $\theta_1,\theta_2\in\{-1,1\}$.
Then there do not exist a rank-$k$ orthogonal projection $Q$ and $\theta_3 \in \{-1,1\}$ such that $T(\theta_1 S(P_1))\cup T(\theta_2 S(P_2))\subseteq \theta_3S(Q)$.
\end{lemma}

\begin{lemma}\label{section3:L:single_cone}
Let $A \in \bH_n$.  Suppose $w_k(A)$ is attained at a unique (rank-$k$) projection $P$ and $w_k(A) = \tr (AP)$.  Suppose $\cC$ is a cone with affine dimension greater or equal to $\gamma^{\bH}_n(k) = k^2 + (n-k)^2$ satisfying $A \in \cC \subseteq \mathcal{P}(A)$.  Then $\cC \subseteq S(P)$, so its affine dimension equals $\gamma^{\bH}_n(k)$.
\end{lemma}

Now, we focus on bijective $\IR$-linear maps preserving parallel pairs on $\bH_n$. We first observe that for such a map $T$,
$T(I)$ is a scalar multiple of $I$; this follows from the following lemma.

\begin{lemma} \label{section3:lem2-new}
Let $A \in \bH_n$. The following are equivalent.
\begin{itemize}
\item[{\rm (a)}] $A = \gamma I$ for some $\gamma \in \IR$.
\item[{\rm (b)}]
$A\| B$ for every $B \in \bH_n$.
\item[{\rm (c)}] $A \| P$  for every rank-$k$ orthogonal projection $P \in \bH_n$.
\end{itemize}
\end{lemma}

Lemma \ref{lem6a--new} also holds with a slight modification of the proof as shown below.

\begin{lemma}\label{section3:lem6a--new} Let $T\colon \bH_n \rightarrow \bH_n$ be a bijective $\IR$-linear map preserving parallel pairs with respect to the $k$-numerical radius. Suppose $A\in\bH_n$ attains its $k$-numerical radius at a unique projection $Q$. Then $X=T^{-1}(A)$ also attains its $k$-numerical radius at a unique projection $P$, or, if $n=2k$, at $P$ and $I-P$ only.
\end{lemma}

\begin{proof} Suppose  $X$ attains its $k$-numerical radius at two distinct rank-$k$ projections $P_1$, $P_2$. Then, by Lemma~\ref{section3:lem5-new},  $X$ belongs to two cones $\mu_1 S(P_1)$, $\mu_2 S(P_2) \subseteq{\mathcal P}(X)$ for some $\mu_1, \mu_2 \in \{1, -1\}$, and each of these cones has, by Lemma~\ref{section3:lem4-new}, affine dimension
	$$\gamma^{\bH}_n(k):=k^2+(n-k)^2.$$
By Lemma~\ref{section3:lem-2proj}, whenever $2k\neq n$, or $2k=n$ and $P_2\neq I-P_1$, their $\IR$-linear spans are  different. Clearly  $T$ maps both cones into cones $${\mathcal C}_i:=T(\mu_iS(P_i))\ni T(X)=A,$$
 contained in $T({\mathcal P}(X))\subseteq {\mathcal P}(T(X))={\mathcal P}(A)$.  However, $A$ achieves its $k$-numerical radius at a unique rank-$k$ projection $Q$, and hence, by Lemma~\ref{section3:L:single_cone},  both ${\mathcal C}_i$ are contained in $S(A)$. Then bijective $\IR$-linear $T$ would map the space $\span_{\IR}(S(P_1)\cup S(P_2))$, whose dimension is greater than $\gamma^{\bH}_n(k)$, onto  $\span_{\IR}({\mathcal C}_1\cup{\mathcal C}_2)\subseteq \span_{\IR}(S(A))$, a  contradiction.

We deduce that $X$ attains its $k$-numerical radius at a single rank-$k$ projection $P$ or, if $n=2k$, only at $P$ and $I-P$. \end{proof}

We now focus on proving Theorem \ref{main:Hermitian}.  There are two main steps.

\textbf{Step 1:} We first prove that the inverse map of a bijective unital $\mathbb R$-linear map preserving parallel pairs with respect to $k$-numerical radius on $\bH_n$ preserves commuting pairs.

\textbf{Step 2:} Theorem \ref{main:Hermitian} holds for $n\geq 3$ and under the extra assumption that $T^{-1}$ preserves commuting pairs.

To prove Step 1, the following lemma will turn out to be very useful.

\begin{lemma}\label{lem:combinatorial}
  Suppose we have  an $s$-tuple ($s\le n$) of positive integers ${\bf n}=(n_1,\dots,n_s)$  satisfying $n_1+\dots+n_s=n$. Then, with  integer $k\in\{1,\dots,n-1\}$ kept fixed,
    $$|\Omega_{\bf n}|:=|\{\{i_1,\dots,i_r\}:\;\ n_{i_1}+\dots+n_{i_r}=k\}|\le {n\choose k};$$
    equality holds if and only if $s=n$ and $n_1=\dots=n_s=1$.
\end{lemma}

\begin{remark}\label{rem:remark-to-combinatorial}
  The same proof can be adapted to  also show that, if $n=2k\ge 4$, then $$|\widehat{\Omega}_{\bf n}|:=|\{\{\{i_1,\dots,i_r\},\{j_1,\dots,j_{s-r}\}\}:\;\ n_{i_1}+\dots+n_{i_r}=n_{j_1}+\dots+n_{j_{s-r}}=k\}|\le\frac{1}{2} {n\choose k}$$ with equality if and only if $s=n$ and $n_1=\dots=n_s=1$.
\end{remark}

Before proving this, we illustrate it with an example.
\begin{example}
  Consider $(k,n)=(9,10)$ and an $s=7$-tuple ${\bf n}=(2,2,2,1,1,1,1)$, that is,  $n_1=n_2=n_3=2$ and $n_4=n_5=n_6=n_7=1$.  Then, $$\Omega_{\bf n}=\{\{1,2,3,4,5,6\},\{1,2,3,4,5,7\},\{1,2,3,4,6,7\},\{1,2,3,5,6,7\}\}.$$ The first collection is in $\Omega_{\bf n}$ because $n_1+n_2+n_3+n_4+n_5+n_6=2+2+2+1+1+1=9=k$; likewise for the other collections.
\end{example}
\begin{proof}  If $n_1=\dots=n_s=1$, then there are clearly ${s \choose k}$ possibilities to select $k$ items among the $s$ available, which gives the sufficiency.
Suppose $n_1\ge 2$ and consider another collection of  $t=s+n_1$  integers ${\bf m}=(m_1,\dots,m_t):=(1,\dots,1,n_2,\dots,n_s)$ (so we split $n_1$ into units and keep every other integer from the original collection).

 With every selection $n_{i_1}+\dots+n_{i_r}=k$, $2\le i_1<\dots<i_r$ (of integers from ${\bf n}$) we associate the same selection of integers from ${\bf m}$ and with every selection $n_{1}+n_{i_2}+\dots+n_{i_r}=k$, $2\le i_2<\dots<i_r$ we associate a selection $(m_1+\dots+m_{n_1})\,+\,(n_{i_2}+\dots+n_{i_r})=k$ of integers from ${\bf m}$.

 This association is one-to-one, but perhaps not surjective. So
 $$|\Omega_{\bf m}|\ge |\Omega_{\bf n}|.$$
 Proceeding recursively, at the one-but-last stage we end up with  a sequence of integers
    ${\bf o}=(o_1,\dots,o_p)$, $o_1+\dots+o_p=n$, where exactly one integer is greater than $1$. Without loss of generality we assume $o_1\ge 2$. As shown above, $|\Omega_{\bf o}|\ge |\Omega_{\bf n}|$.  With one additional step we end up with a sequence ${\bf z}=(1,\dots,1)$ of $n$ units.
    Now, $|\Omega_{\bf z}|={n\choose k}$, and the above one-to-one  association between elements from $\Omega_{\bf o}$ and $\Omega_{\bf z}$ cannot be surjective: Namely, the collection of indices in $\Omega_{\bf o}$ corresponding to  a selection $o_{i_1}+\dots+o_{i_r}=n$ with $2\le i_1<\dots<i_r$ associates to itself, while for every selection with $i_1=1$ we have, besides the associated selection $(1+\dots+1)+ o_{i_2}+ \dots +o_{i_r}$ (of integers in ${\bf z}$) also a new selection, namely the one with indices $\{1, (n-k+2),\dots,n\}\in\Omega_{\bf z}$ (or $\{1\}$ if $k=1$).
     \end{proof}

\begin{lemma}
Let $1\le k<n$. Suppose a unital, bijective, $\IR$-linear $T\colon\bH_n\to\bH_n$ preserves parallel pairs with respect to the $k$-numerical radius $w_k$ in one direction. Then, $T^{-1}$ preserves commutativity.
\end{lemma}
\begin{proof}
Note that $A,B\in\bH_n$ commute if and only if they are simultaneously diagonalizable under conjugation by a unitary matrix. By  absorbing the corresponding unitary conjugation inside $T$ we can assume that $A,B$ are already diagonal. The result will then follow once we show that $X_j:=T^{-1}(E_{jj})\in\bH_n$ are commuting.

We shall heavily use Lemmas~\ref{section3:lem5-new}--\ref{section3:lem4-new}, that is, if $X\in\bH_n$ attains its norm at a unique rank-$k$ projection $P$ (or, when $n=2k$, at $P$ and possibly  $I-P$) then ${\mathcal P}(X)=S(P)\cup (-S(P))$ (or, when $n=2k$, possibly ${\mathcal P}(X)=S(P)\cup (-S(P))\cup S(I-P)\cup (-S(I-P))$),  so for each $k\in\{1,\dots,n-1\}$,  $\span_{\IR}{\mathcal P}(X)=\span_{\IR} S(P)=\bH(\mathrm{Im} P)\oplus\bH(\Ker P)$, where $\bH(V)$ denotes the $\IR$-subspace of self-adjoint operators on a vector space $V$. We shall also use the following claim.

{\it Claim}: If $Q$ is a rank-$k$ diagonal projection, and ${\mathcal V}=\bigoplus_{j=1}^s\bH_{n_j}$, then $(\span S(Q))\cap {\mathcal V}={\mathcal V}$ if and only if the compressions of $Q$ to diagonal blocks $\bH_{n_j}$ are $0$ or $I_{n_j}$.

{\it Proof of Claim}: Recall that $\span S(Q) =\bH(\mathrm{Im}(Q))\oplus \bH(\mathrm{Ker}(Q))$. Assume the compression $Q_1$  of $Q$ to $\bH_{n_1}$ is neither identity nor zero. Then there exist normalized vectors  $x_1,y_1\in\IC^{n_1}$ such that $Q_1x_1=x_1$ and $ Q_1y_1=0$. Let $x:=x_1\oplus 0\in\IC^n$ and likewise $y:=y_1\oplus 0$. It  follows that $xy^\ast+yx^\ast \in \bH_{n_1}\oplus 0\subseteq{\mathcal V}$, however, it does not belong to  $\bH(\mathrm{Im}(Q))\oplus \bH(\mathrm{Ker}(Q))$, so it lies outside of $\span S(Q)$. In particular, $\span S(Q)\cap {\mathcal V}$ is properly contained in ${\mathcal V}$. The other direction is easy (and actually will not be required in the proof).

By using Lemma \ref{section3:lem6a--new}, we now note that, for each $i_1<\dots<i_k$, the matrix $\hat{X}_{i_1,\dots,i_k}:=X_{i_1}+\dots+X_{i_k}$ either attains its $k$-numerical radius at a unique rank-$k$ projection $P_{i_1,\dots,i_k}$, or $n=2k$ and it attains its $k$-numerical radius at only $P_{i_1,\dots,i_k}$ and $I-P_{i_1,\dots,i_k}$. In either of these two cases, Lemmas~\ref{section3:lem5-new} and~
\ref{section3:lem4-new}
give $$\span_{\IR}{\mathcal P}(\hat{X}_{\bf i})=\bH(\mathrm{Im}\,P_{\bf i})\oplus \bH(\Ker P_{\bf i})$$ for each multiindex ${\bf i}=(i_1,\dots,i_k)$. Note that the rank-$k$ projections for distinct multiindices~${\bf i}$ and~${\bf j}$ must be pairwise distinct (except when $n=2k$ and $\{\bf i\}\cup\{\bf j\}=\{1,\dots,n\}$), for otherwise, due to $\span_{\IR}S(P_{\bf i})=\span_{\IR}{\mathcal P} (\hat{X}_{\bf i})$,
$T$ would map $\span_{\IR}S(P_{\bf i})=\span_{\IR} S(P_{\bf j})$ bijectively onto $\span_{\IR}{\mathcal P}(E_{i_1i_1}+\dots+E_{i_ki_k}) $ and onto $\span_{\IR}{\mathcal P}(E_{j_1j_1}+\dots+E_{j_kj_k}) $, respectively, which are two different spaces.

Choose a multiindex ${\bf i} =(i_1,\dots,i_k)$.
There exists a unitary
$U_{\bf i}$ such that $U_{\bf i}\hat{X}_{{\bf i} } U_{\bf i}^* =
\hat{Y}_{{\bf i}} \oplus \hat{Y}'_{{\bf i}}$ with $w_k(\hat{Y}_{{\bf i} }) \geq w_k(\hat{Y}'_{{\bf i} })$ (and equality is possible only when $n=2k$). Also, there exists a permutation matrix $Z_{\bf i}$ such that $E_{i_1i_1}+\dots+E_{i_ki_k}=Z_{\bf i}^\ast(E_{11}+\dots+E_{kk})Z_{\bf i}$.
It follows that
$$
U_{\bf i}\hat{X}_{{\bf i} }U_{\bf i}^\ast\in S(E_{11}+\dots+E_{kk})
\subseteq {\mathcal P}(U_{\bf i}\hat{X}_{{\bf i} }U_{\bf i}^\ast)
=U_{\bf i}{\mathcal P}(\hat{X}_{{\bf i} })U_{\bf i}^\ast,$$
and hence 
\begin{eqnarray*}
Z_{\bf i}^\ast(E_{11}+\dots+E_{kk})Z_{\bf i}&=&E_{i_1i_1}+\dots+E_{i_ki_k}= T(\hat{X}_{{\bf i} }) \in T(U_{\bf i}^\ast S(E_{11}+\dots+E_{kk})U_{\bf i})\\
&\subseteq&
T({\mathcal P}(\hat{X}_{{\bf i} }))\subseteq {\mathcal P}(T(\hat{X}_{{\bf i} }))
= {\mathcal P}(E_{i_1i_1}+\dots+E_{i_ki_k})\\
&=&Z_{\bf i}^\ast{\mathcal P}(E_{11}+\dots+E_{kk})Z_{\bf i}
\end{eqnarray*}
Conjugate on both sides by $Z_{\bf i}$ and apply Lemma~\ref{section3:L:single_cone}, with $A:=E_{11}+\dots+E_{kk}$,  on a cone
${\mathcal C}:=Z_{\bf i}T(U_{\bf i}^\ast S(E_{11}+\dots+E_{kk})U_{\bf i})Z_{\bf i}^\ast$,
of  affine dimension $\gamma^{\bH}_n(k)$, contained in ${\mathcal P}(A)$ to get
$\mathcal C \subseteq S(E_{11}+\dots+E_{kk})$.  Taking linear spans, we have
\begin{equation*}
\begin{aligned}
    T(U_{\bf i}^\ast &(\bH_k\oplus\bH_{n-k}) U_{\bf i}) =T(\span_{\IR}(U_{\bf i}^\ast S(E_{11}+\dots+E_{kk})U_{\bf i})) \\
    &= \span_{\IR}T(U_{\bf i}^\ast S(E_{11}+\dots+E_{kk})U_{\bf i})
    \subseteq \span_{\IR} Z_{\bf i}^\ast \mathcal{P} (E_{11}+\dots+E_{kk}) Z_{\bf i}=Z_{\bf i}^\ast (\bH_k\oplus\bH_{n-k}) Z_{\bf i}.
\end{aligned}
\end{equation*}
Since $T$ is an $\IR$-linear bijection, we must actually have 
$U_{\bf i}^\ast (\bH_k\oplus\bH_{n-k}) U_{\bf i}=T^{-1}(Z_{\bf i}^\ast (\bH_k\oplus\bH_{n-k}) Z_{\bf i})$.
Finally, conjugation by the permutation matrix $Z_{\bf i}$ permutes the diagonal matrices among themselves, so $E_{jj}\in Z_{\bf i}^\ast (\bH_k\oplus \bH_{n-k}) Z_{\bf i}^\ast$, and hence
\begin{equation}\label{eq:X_j(a)}U_{\bf i}\hat{X}_{{\bf i} }U_{\bf i}^\ast\in S(E_{11}+\dots+E_{kk})\text{ and } X_j=T^{-1}(E_{jj})\in U_{\bf i}^\ast(\bH_k\oplus\bH_{n-k})U_{\bf i}\text{ for all } j=1,\dots,n.\end{equation}
 
Now, apply this to the multiindex ${\bf i}_0=(1,\dots,k)$. We may assume that $U_{{\bf i}_0}=I_n$; then $P_{{\bf i}_0 }=E_{11}+\dots+E_{kk}$ and

$$X_j\in \bH_k\oplus \bH_{n-k}\quad\hbox{ for }j=1,\dots,n;$$
so consequently also
$$\hat{X}_{\bf i}\in\bH_k\oplus \bH_{n-k}\text{ for every multiindex }{\bf i}.$$
If $(n,k)=(2,1)$ we are done. We proceed assuming $n\ge 3$.

Choose a multiindex ${\bf i}_1:=(i_1,\dots,i_k)\neq {\bf i}_0$.  Recall that  $\hat{X}_{{\bf i}_1}$ achieves its $k$-numerical radius norm at a unique rank-$k$ projection (again, except if $n=2k$ when it   might achieve the $k$-numerical radius also at its orthogonal complement) which coincides with the eigenprojection  to  the  $k$-tuple  of its eigenvalues whose sum is  maximal in modulus. Now  this norm-achieving $k$-tuple of eigenvalues consists of $k_1$ eigenvalues that belong to the first block and $k_2$ eigenvalues that belong to the second block. Then there exists a unitary $U_1=U_1'\oplus U_1''\in \bM_k\oplus \bM_{n-k}$ such that $U_1\hat{X}_{{\bf i}_1} U_1^\ast$  achieves its $k$-numerical radius at a diagonal rank-$k$ projection $P_{{\bf i}_1}=(E_{11}+\dots+E_{k_1k_1})+ (E_{(n-k_2+1)(n-k_2+1)}+\dots+E_{nn})$. Notice that, due to their block-diagonal structure, we still have that  $U_1\hat{X}_{{\bf i}_0}U_1^\ast$ achieves its $k$-numerical radius   at $E_{11}+\dots+E_{kk}$.

We can hence also assume that $U_1=I_n$.
Let  now $S_{{\bf i}_1}$ be the permutation matrix which permutes the last $k_2$ rows/columns with rows/columns indexed by $(k_1+1),\dots,k$.
Then \eqref{eq:X_j(a)} holds for two multiindices  ${\bf i}_1=(i_1,\dots,i_k)$ (and unitary $U_{{\bf i}_1}=S_{{\bf i}_1}$)  as well as for ${\bf i}_0=(1,\dots,k)$ (and unitary $U_{{\bf i}_0}=I_n$) and gives $$X_j\in S_{\bf i_1}^\ast (\bH_k\oplus\bH_{n-k})S_{\bf i_1} \cap (\bH_k\oplus\bH_{n-k});\qquad j=1,\dots,n$$
One can check that this is equivalent to
\begin{equation}\label{eq:X_j(iteration2)}
  X_j\in \bH_{k_1}\oplus\bH_{k-k_1}\oplus \bH_{n-k-k_2}\oplus \bH_{k_2};\qquad j=1,\dots,n.
\end{equation}
Thus, we may repeat the above procedure to see that every $\hat{X}_{\bf j}$ achieves its $k$-numerical radius at a unique rank-$k$ projection (and perhaps its orthogonal complement, if $n=2k$) in $\bH_{k_1}\oplus\bH_{k-k_1}\oplus \bH_{n-k-k_2}\oplus \bH_{k_2}$ . By conjugating with a  suitable block-diagonal unitary $U_{\bf j}$ in $\bM_{k_1}\oplus\bM_{k-k_1}\oplus \bM_{n-k-k_2}\oplus \bM_{k_2}$ we can achieve that this projection is the sum of $k$ diagonal matrix units $E_{tt}$, while the rank-$k$ projections associated with previously treated $\hat{X}_{{\bf i}_0}$ and $\hat{X}_{{\bf i}_1}$ remain the same. We can again assume that $U_{\bf j}=I_n$; then the equation~\eqref{eq:X_j(iteration2)} will be still valid but with more blocks of smaller size.

We can now recursively repeat the above procedure for all $n\choose k$ different matrices $\hat{X}_{\bf i}$ which have pairwise different  associated rank-$k$ projections (or, if $n=2k$, we consider the half as many pairwise distinct associated pairs of projections and their orthogonal complement);
each time either at least one block in the currently obtained version of~\eqref{eq:X_j(iteration2)} will split into two smaller blocks or else, by the Claim, the associated rank-$k$ projection will be the sum of identities in some of the blocks.

By the  pigeonhole principle it is impossible that in the final  intersection, the obtained equivalent of \eqref{eq:X_j(iteration2)}, that is,
\begin{equation}\label{eq:final}
  X_j\in\bigoplus_i^s \bH_{n_i};\qquad j=1,\dots,n
\end{equation}
would have some block   of size bigger than $1$-by-$1$. Namely, every associated rank-$k$ projection of $\hat{X}_{j_1,\dots,j_k}$ is either zero or identity in each block and as such the number of  different associated  rank-$k$ projections match the  number of  collection of indices $\{i_1,\dots,i_r\}$ satisfying $n_{i_1}+\dots+n_{i_r}=k$. By Lemma~\ref{lem:combinatorial} (if $n=2k\ge 4$ we use Remark~\ref{rem:remark-to-combinatorial} instead) this number is strictly below ${n\choose k}$ except if $s=n$ and $n_1=\dots=n_s=1$, i.e, except when in~\eqref{eq:final} all the blocks are $1$-by-$1$.

Thus, by a recursive sequence of unitary conjugations we achieve that all $X_j$ are diagonal, and hence commute.
\end{proof}

\begin{lemma}\label{lem:S(P)-is-pointed-cone}
	Let $n=2k$ and let $P \in \bH_n$ be a rank-$k$ projection.  Then $S(P) \cup -S(P)$ cannot contain a two-dimensional subspace.
\end{lemma}

\begin{proof}
	Without loss of generality we may assume $P = I_k \oplus 0_k$.  Suppose $A, B \in \bH_n$ are nonzero matrices whose span is contained in $S(P) \cup -S(P)$.  Without loss of generality we may suppose $A, B \in S(P)$.  We may write $A=A_1 \oplus A_2$, $B = B_1 \oplus B_2$ with $\tr A_1 = w_k(A) >0$ and $\tr B_1 = w_k(B) > 0$.  Choose $t < 0$ such that $\tr A_1 + t \tr B_1 = 0$.  Since $A+tB \in S(P) \cup -S(P)$, we must have $0 = |\tr (A_1 + tB_1)| = w_k (A+tB)$, so $A+tB = 0$.  Thus $A$ and $B$ are linearly dependent.
\end{proof}

\begin{lemma} \label{commute}
	Let $n \geq 3$ and $1 \leq k < n$. Suppose $T\colon\bH_n \to \bH_n$ is a unital surjective linear map that preserves parallel pairs with respect to the $k$-numerical radius and $T^{-1}$ preserves commutativity.  Then there exists a unitary $U$ such that
	\begin{enumerate}
		\item[{\rm (a)}] $T^{-1}$ has the form $A \mapsto U^*AU$ or $A \mapsto U^*A^tU$, or
		\item[{\rm (b)}] $n=2k$ and there exists a nonzero $c \in \IR$ such that 
		$T^{-1}(A) = cU^*AU$ for all $A \in \bH_n^0$, \ or \  $T^{-1}(A) = cU^*A^tU$ for all $A \in \bH_n^0$, where $\bH_n^0$ denotes the set of trace zero matrices in $\bH_n$.
	\end{enumerate}
\end{lemma}

\begin{proof}
Because $T^{-1}$ is a linear surjective map preserving commutativity, by \cite[Theorem 2]{CJR}, there exists a unitary matrix $U$, a linear functional $f$ on $\bH_n$, and a nonzero scalar $c$ such that $T^{-1}$ has one of the following forms:
	\begin{enumerate}
		\item $T^{-1}(A) = cU^* A U + f(A) I$ for all $A \in \bH_n$, or
		\item $T^{-1}(A) = cU^* A^t U + f(A) I$ for all $A \in \bH_n$.
	\end{enumerate}
Consider the first case; the second case is similar.  Since $T$ is unital, $I = cI + f(I) I$, so $1 = c  + f(I)$.  For all $r \in \IR$,
\begin{equation*}
  \begin{aligned}
		T^{-1}(A+rI) &= cU^*AU + (cr + f(A) + r f(I)) I = cU^*AU + (r + f(A)) I \\
&= c(U^*AU + \frac{1}{c} (r + f(A)) I )
	\end{aligned}
\end{equation*}
Now suppose $A$ has distinct eigenvalues $\lambda_1(A) > \dots > \lambda_n(A)$ and let
$$s =  -\frac{1}{2k} \left( \sum_{i=1}^k \lambda_i(A) + \sum_{i=1}^k \lambda_{n+1-i}(A) \right).$$
Then $A+rI$ attains its numerical radius at a unique rank-k projection for all $r \ne s$.  Consider the case $n \ne 2k$.  Then Lemma \ref{section3:lem6a--new} implies that $T^{-1}(A+rI)$ also attains its numerical radius at a unique rank-k projection for all
$r \ne s$.   But $U^*AU + tI$ attains its $k$-numerical radius at a unique rank-$k$ projection for all $t \ne s$, so
$$s = \frac{1}{c}(s + f(A)).$$
Thus
$$f(A) = s (c-1) = -s f(I) = \frac{f(I)}{2k}  \left( \sum_{i=1}^k \lambda_i(A) + \sum_{i=1}^k \lambda_{n+1-i}(A) \right);$$
by continuity, this holds for all $A \in \bH_n$.
If $f$ is the zero functional the result follows.  Now suppose $f(I) \ne 0$.  This implies that the map $g$ defined by
$$g(A) = \left( \sum_{i=1}^k \lambda_i(A) + \sum_{i=1}^k \lambda_{n+1-i}(A) \right)$$
is linear, a contradiction. 
(For example, when $n > 2k$, observe that $g(A+B) \ne g(A) + g(B)$ when $A = I_k \oplus (-I_k) \oplus 0_{n-2k}$ and
$B = I_k \oplus 0_1 \oplus (-I_k) \oplus 0_{n-2k-1}$; while if $n<2k$ we consider $A=I_k\oplus 0_{n-k}$ and $B=0_{n-k}\oplus I_k$.)

Now let $n = 2k$.  It suffices to prove that $f$ is a scalar multiple of the trace functional $X\mapsto\Tr(X)$. Note that
the inverse of the unital linear map $T^{-1}\colon Y\mapsto c U^\ast YU+ f(Y) I$ is
$$T(X)=Y=\frac{1}{c} UXU^\ast -\frac{1}{c}f(UXU^\ast)I.$$ To see this, start with $X=T^{-1}(Y)=cU^\ast YU+f(Y)$, then $Y=\frac{1}{c} UXU^\ast -\frac{1}{c}f(Y)I$; applying $f(\cdot)$ and using $f(I)=1-c$
we deduce that $T(X)=Y=\frac{1}{c} UXU^\ast -\frac{1}{c}f(UXU^\ast)I$. Define $\tilde{T}$ by $\tilde{T}(X) = cU^\ast T(X)U$, so $\tilde{T}$ preserves parallel pairs and
$$\tilde{T}(X)=X+g(X)I,$$ where $g(X):=-f(UXU^\ast)$.

Suppose, by way of contradiction, that $g$ is not a scalar multiple of
the trace functional.  Then there exists a Hermitian  $X\in\Ker \Tr\setminus\Ker g$  with all its eigenvalues simple (since the set of trace-zero Hermitian matrices with all eigenvalues simple is dense in $\Ker\Tr\cap\bH_n$). We can assume that $X=X_1\oplus X_2$, $X_1=\diag(\lambda_1(X),\dots,\lambda_k(X))$, $X_2=\diag(\lambda_{k+1}(X),\dots,\lambda_n(X))$, and $\Tr(X_1)=-\Tr(X_2)=w_k(X)$. Since all eigenvalues of $X$ are simple, it attains its $k$-numerical radius at exactly two rank-$k$ projections, $P=E_{11}+\dots+E_{kk}$ and $I-P$. As such, ${\mathcal P}(X)=(-S(P))\cup S(P)\cup (-S(I-P))\cup S(I-P)$ contains a two-dimensional $\IR$-subspace, namely $\IR I_k\oplus \IR I_k=\IR(I_k\oplus 0_k) + \IR(0_k\oplus I_k)$.
   This is mapped into $\tilde{T}(X)=X+g(X)I$ which achieves its $k$-numerical radius at exactly one rank-$k$ projection (if
    $g(X)>0$ this is $P$; if $g(X)<0$ this would be $I-P$).  But then $(- S(P))\cup S(P)={\mathcal P}(\tilde{T}(X))\supseteq \tilde{T}({\mathcal P}(X))$ would contain a two-dimensional subspace, contradicting  Lemma~\ref{lem:S(P)-is-pointed-cone}.
Thus $g$ is a scalar multiple of the trace functional, and hence so is $f$.
\end{proof}

\begin{proof}[Proof of Theorem \ref{main:Hermitian}]
The implication (b) $\Rightarrow$ (a) is clear.
The implication (a) $\Rightarrow$ (c) follows from Lemma \ref{commute} if $n \ge 3$, so suppose (a) holds and $n = 2$.  We may assume that $T$ is unital by Lemma \ref{section3:lem2-new}.
Suppose a nonzero $X \in \bH_2$ has trace zero.  Then $X$ attains its numerical radius at exactly two rank-1 projections $P$ and $I-P$.  Thus $\mathcal{P}(X) = S(P) \cup (-S(P)) \cup (-S(I-P)) \cup S(I-P)$ contains a two-dimensional subspace $\IR P+\IR(I-P)$.
Since $T(\mathcal{P}(X)) \subseteq \mathcal{P}(T(X))$ and $T$ is bijective, $\mathcal{P}(T(X))$ contains a two-dimensional subspace.  It follows that $\Tr T(X) = 0$.  (Reason: If $\Tr T(X) \ne 0$, then $T(X)$ (which can't be a scalar multiple of $I$ due to the bijectivity of $T$) must attain its numerical radius at a unique projection $Q$, so $\mathcal{P}(T(X)) = S(Q) \cup -S(Q)$ contains a two-dimensional subspace, contradicting Lemma~\ref{lem:S(P)-is-pointed-cone}.)  It follows that $T$ preserves the set of trace-zero matrices; since $T$ is also unital, it preserves the trace and (c.3) holds.

It remains to prove (c) $\Rightarrow$ (b).
 If (c.1) holds, then clearly (b) holds.

Suppose (c.2) holds. We may replace $T$ by $A \mapsto \gamma UT(A^+)U^*$ for a suitable nonzero $\gamma$ and unitary $U$ (here $A^+$ denotes $A$ or $A^t$, depending on the case in (c.2)) and assume that $T$ is unital and $T(X) = cX$ for all $X \in \bH_n^0$.  Now suppose $(A_1, A_2)$ is a nontrivial TEA pair (that is, $A_1$, $A_2$ are nonzero).  If $A_1 + A_2$ attains its $k$-numerical radius at a rank-$k$ projection $P$ then
$$|\tr A_1P | + |\tr A_2P| \leq w_k(A_1) + w_k(A_2) = w_k(A_1+A_2) = |\tr (A_1+A_2)P| \leq |\tr A_1P| + |\tr A_2P|,$$
so $w_k(A_j) = |\tr A_jP|$ and
$\tr A_1P$, $\tr A_2P$ have the same sign. Moreover, $\Ran P$ is a reducing subspace for both $A_1$ and $A_2$. If needed, we may replace $A_j$ by $-A_j$ and assume $\tr A_j P > 0$.

We may write $A_j = a_j I  + Z_j$ where $\tr Z_j = 0$.  With respect to the decomposition $\IC^n = \Ran P \oplus \Ker P$, we may further write $Z_j = X_j \oplus Y_j$ with $\tr Y_j = - \tr X_j$.  Then $w_k(A_j) = a_j k + \tr X_j$ and $a_j, \tr X_j \geq 0$ (if they had opposite signs, $|\tr A_j(I-P)| = |a_j k - \tr X_1| > |a_j k + \tr X_1|  = \tr A_j P = w_k(A_j)$, a contradiction).

By Proposition \ref{prop:propertiesofW_k}(2) it follows that $\tr X_j$ is the sum of the $k$ largest eigenvalues of $Z_j$.  Thus $T(A_j) = a_jI + cZ_j$ attains its $k$-numerical radius at $P$ if $c > 0$ and at $I-P$ if $c<0$, so $T(A_1), T(A_2)$ form a TEA pair.  Thus (b) holds.

Finally, suppose $n=2$ and (c.3) holds.  We may scale $T$ and assume $T$ is unital.  Suppose $(A_1,A_2)$ is a nontrivial TEA pair.  As above when we assumed (c.2), we may write $A_j = a_j I + x_j \left(\begin{smallmatrix} 1 & 0 \\ 0 & -1 \end{smallmatrix}\right)$ with respect to a suitable basis. Write $X_{0} = \left(\begin{smallmatrix} 1 & 0 \\ 0 & -1 \end{smallmatrix}\right)$.  Since $T$ preserves trace, we may replace $T$ by $A \mapsto UAU^*$ for a suitable unitary and assume that $T(X_0) = cX_0$ for some nonzero $c$.  Then the argument for (c.2) applies and we conclude that $T(A_1), T(A_2)$ form a TEA pair, so (b) holds.
\end{proof}

\bigskip
\noindent
{\bf \Large Acknowledgment}\nopagebreak

This  research was supported in part by the Slovenian Research Agency (research program P1-0285, research project N1-0210, and bilateral project BI-US-22-24-129).
Li is an affiliate member of the Institute for Quantum Computing, University of Waterloo; his research was
also supported by the Simons Foundation Grant 851334. Singla is a PIMS Postdoctoral Fellow in the University of Regina; his research was also co-funded by PIMS and the Department of Mathematics and Statistics, University of Regina.

\bigskip
\noindent
(Kuzma)
Department of Mathematics, University of Primorska, Slovenia; Institute of Mathematics,
Physics,
and Mechanics, Slovenia. E-mail: bojan.kuzma@famnit.upr.si

\medskip\noindent
(Li) Department of Mathematics, College of William \& Mary, Williamsburg, VA 23187, USA.
E-mail: ckli@math.wm.edu

\medskip\noindent
(Poon)
Department of Mathematics, Embry-Riddle Aeronautical University, Prescott AZ 86301, USA.
E-mail: poon3de@erau.edu

\medskip\noindent
(Singla) 
The Department of Mathematics and Statistics, University of Regina, Canada. E-mail: ss774@snu.edu.in


\begin{thebibliography}{WW}






\bibitem{CJR} M.D. Choi, A.A. Jafarian, and H. Radjavi, Linear maps preserving commutativity, Linear Algebra Appl.\
87 (1987) 227-241.


\bibitem{GS} P. Grover and S. Singla, Subdifferential set of the joint numerical radius of a tuple of matrices, Linear Multilinear
Algebra 71 (2023), 2709--2718.


\bibitem{KLPS} B. Kuzma, C.K. Li, E. Poon, S. Singla, Linear maps preserving parallel matrix pairs
with respect to the Ky-Fan $k$-norm, Linear Algebra Appl.\ 687 (2024), 68--90.

\bibitem{Li} C.K. Li,
Matrices with some extremal properties, Linear Algebra Appl. 101 (1988), 255-267.

\bibitem{Li87} C.K. Li, Linear operators preserving the higher numerical radius of matrices, Linear  Multilinear Algebra
21 (1987), 63--73.

\bibitem{Li94}
C.K. Li, $C$-numerical ranges and $C$-numerical radii, Linear  Multilinear Algebra
37 (1994), 51--82.

\bibitem{Lietal} C.K. Li, Y.T. Poon, N.S. Sze,
Linear maps transforming the higher numerical ranges, Linear Algebra Appl.\ 400 (2005), 291-311.

 

\bibitem{LPW} C.K. Li, Y.T. Poon, and Y.S. Wang, Joint numerical ranges and commutativity of matrices, J. Math. Anal. Appl.\
491 (2020) 124310.

\bibitem{LTWW0} C.K. Li, M.C. Tsai, Y. Wang, and N. Wong,
Linear maps preserving $\ell_p$-norm parallel vectors,
 arXiv:2407.19276.


\bibitem{LTWW} C.K. Li, M.C. Tsai, Y. Wang, and N. Wong, Linear maps preserving
parallel pairs,  arXiv:2408.06366.


\bibitem{LiTsing88a} C.K. Li, N.K. Tsing, Linear operators that preserve the $c$-numerical range
or radius of matrices, Linear and Multilinear Algebra \textbf{23} (1988), 27--46.

 

\bibitem{Marsh-Olkin} A.W. Marshall, I. Olkin, B.C. Arnold, Inequalities: Theory of Majorization and Its Applications, second edition,
Springer, New-York, 2011.

\bibitem{Westwick} R. Westwick, A theorem on numerical range, Linear and Multilinear Algebra 2 (1975), 311--t315.

\bibitem{Zamani} A. Zamani,  Characterization of numerical radius parallelism in $C^*$-algebras. Positivity 23(2)
(2019), 397--411.

\end{thebibliography}
\end{document}